\documentclass[12pt,reqno]{amsart}
\usepackage{eufrak}
\usepackage{amsmath,amsthm,amsfonts,amssymb,times}
\usepackage{graphicx,pgfarrows,pgfnodes}
\usepackage[mathscr]{eucal}
\usepackage{verbatim}
\usepackage{url,cite,mleftright}
\usepackage{tkz-fct}
\usepackage{tikz}
\usepackage[margin=1in]{geometry}
\usepackage{stackrel}
\usepackage{tikz-cd}
\usepackage{colonequals}
\usepackage{float}
\restylefloat{table}
\usepackage{placeins} 

\usepackage{academicons}
\usepackage{xcolor}

\usepackage{lscape}
\usepackage{rotating}
\usepackage[
pdfauthor={},
pdfkeywords={},
pdftitle={},
pdfcreator={},
pdfproducer={},
linktocpage,colorlinks,bookmarksnumbered,linkcolor=blue,
citecolor=red,urlcolor=red]{hyperref}

\definecolor{lime}{HTML}{A6CE39}
\DeclareRobustCommand{\orcidicon}{%
	\begin{tikzpicture}
		\draw[lime, fill=lime] (0,0) 
		circle [radius=0.16] 
		node[white] {{\fontfamily{qag}\selectfont \tiny ID}};
		\draw[white, fill=white] (-0.0625,0.095) 
		circle [radius=0.007];
	\end{tikzpicture}
	\hspace{-2mm}
}

\DeclareMathOperator{\Cl}{Cl}

\DeclareMathOperator{\disc}{disc}
\DeclareMathOperator{\End}{End}

\DeclareMathOperator{\Gal}{Gal}

\DeclareMathOperator{\sgn}{sgn}

\DeclareMathOperator{\El}{Ell}

\DeclareMathOperator{\ord}{ord}

\newcommand{\QED}{\hspace{\stretch{1}} $\blacksquare$}

\newcommand{\CC}{\mathbb{C}}

\newcommand{\HH}{\mathbb{H}}
\newcommand{\NN}{\mathbb{N}}

\newcommand{\QQ}{\mathbb{Q}}
\newcommand{\RR}{\mathbb{R}}
\newcommand{\ZZ}{\mathbb{Z}}

\renewcommand{\(}{\left\(}
\renewcommand{\)}{\right\)}
\renewcommand{\pmod}[1]{\,(\textup{mod}\,#1)}

\theoremstyle{plain}
\newtheorem{thm}{Theorem}
\newtheorem{lemma}[thm]{Lemma}

\newtheorem{cor}[thm]{Corollary}
\newtheorem{prop}[thm]{Proposition}

\theoremstyle{definition}

\newtheorem{ex}[thm]{Example}

\theoremstyle{remark}
\newtheorem{rem}[thm]{Remark}

\numberwithin{equation}{section}
\numberwithin{thm}{section}
\allowdisplaybreaks

\newcommand{\orcidSarth}{\href{https://orcid.org/0000-0002-1724-0305}{\orcidicon}}

\newcommand{\orcid}[1]{\href{https://orcid.org/#1}{\textcolor[HTML]{A6CE39}{\aiOrcid}}}

\begin{document}
	\title[On discriminants of minimal Polynomials of the Ramanujan
	$t_n$ Class Invariants]{On discriminants of minimal Polynomials of the Ramanujan $t_n$ Class Invariants}
	\author[Sarth Chavan]{Sarth Chavan\orcidSarth}
	\thanks{The author was supported by The 2022 Spirit of Ramanujan Fellowship and The Mehta Fellowship}
	\address{$^{1}$ Department of Mathematics, Massachusetts Institute of Technology, Cambridge, MA, USA.}
	\email{schavan@mit.edu}
	\address{$^{2}$ Euler Circle, Palo Alto, California 94306, USA.}
	\email{sarth5002@outlook.com}
	\date{\today}
	\dedicatory{Dedicated to all my Rickoid friends who turned into a family}
	\begin{abstract}
		We study the discriminants of the minimal polynomials $\mathcal{P}_n$ of the Ramanujan $t_n$ class invariants, which are defined for positive $n\equiv11\pmod{24}$. 
		The historical precedent for doing so comes from Gross and Zagier, which is known for computing the prime factorizations of certain resultants and discriminants of the Hilbert class polynomials $H_n$. We show that $\Delta\left(\mathcal{P}_n\right)$  divides $\Delta\left(H_n\right)$ with quotient a perfect square, and as a consequence, we explicitly determine the sign of $\Delta\left(\mathcal{P}_n\right)$ based on the class group structure of the order of discriminant $-n$.  
		We also show that the discriminant of number field generated by $j\left(\frac{-1+\sqrt{-n}}{2}\right)$, where $j$ is the $j$-invariant, divides $\Delta\left(\mathcal{P}_n\right)$.
		Moreover, we show that 3 never divides $\Delta\left(\mathcal{P}_n\right)$ for all  squarefree positive integers $n\equiv11\pmod{24}$. 
		
	\end{abstract}
	\maketitle
\section{Introduction and Main Results}	

Let $E/\CC$ be an elliptic curve $E$ over  $\CC$ that has \emph{complex multiplication} (CM) by an imaginary quadratic order  $\mathcal{O}$, by which we mean that the endomorphism ring $\End(E)$ is isomorphic to $\mathcal{O}$. 
Let $K$ denote the fraction field of $\mathcal{O}$. The $j$-invariant of $E$ is an algebraic
integer whose minimal polynomial over $K$ is the \emph{Hilbert class polynomial $H_n$}\footnote{Note that most authors use the term \emph{Hilbert class polynomial} only when $\mathcal{O}$ is a maximal order (they then use the term \emph{ring class polynomial} for the general case);  however, we will not make this distinction.}, 
where $n$ is the discriminant of
$\mathcal{O}$.

In particular, the Hilbert class polynomial $H_n$ is defined by
\[H_n(x)\colonequals\prod_{j(E)\in\El_{\mathcal{O}}(\CC)}\left(x-j(E)\right),\]
where $\El_{\mathcal{O}}(\CC)\colonequals\{j(E/\CC):\End(E)\cong\mathcal{O}\}$ is the set of $j$-invariants of elliptic curves $E/\CC$ with complex multiplication by the imaginary quadratic order $\mathcal{O}$ with discriminant $n = \disc(\mathcal{O})$.


Moreover, $H_n \in \ZZ[x]$ and its splitting field over the imaginary quadratic field $K$ is the \emph{ring class field} $K_{\mathcal{O}}$, which is an abelian extension of $K$ whose Galois group $\Gal\left(K_{\mathcal{O}}/K\right)$ is isomorphic to the class group $\Cl\left(\mathcal{O}\right)$, via the Artin map.
This is indeed a remarkable result as it implies that of  uncountably many isomorphism classes of elliptic curves over $\CC$, only countably many have CM.

Ramanujan, who made many beautiful and elegant discoveries in his short life of 32 years, defined in his third notebook \cite[Pages 392-393]{SR} the values
\begin{equation*}
	t_n\colonequals\dfrac{f\left(\sqrt[3]{q_n}\right)f\left(q_n^3\right)}{f^2\left(q_n\right)}\sqrt{3}q_n^{1/18},
\end{equation*}
where $q_n=\exp\left(-\pi\sqrt{n}\right)$, and
$	f\left(-q\right)=\prod_{n\geqslant1}\left(1-q^n\right).$

For all positive $n\equiv11\pmod{24}$, let $\mathcal{P}_n$ be the minimal polynomial of $t_n$ over  $\QQ$. 
Without any further explanation on how he found them, Ramanujan gave the following table of polynomials $\mathcal{P}_n$ based on $t_n$ for first five values of $n \equiv11\pmod{24}$:
\begin{table}[h]
		\centering
		\begin{tabular}{|c|c|}
			\hline
			\textbf{$n$} & \textbf{$\mathcal{P}_n(z)$} \\ \hline
			$11$         & $z-1$                       \\ \hline
			$35$         & $z^2+z-1$                   \\ \hline
			$59$         & $z^3+2z-1$                  \\ \hline
			$83$         & $z^3+2z^2+2z-1$             \\ \hline
			$107$        & $z^3-2z^2+4z-1$             \\ \hline
		\end{tabular}
	\vspace{0.1in}
		\caption{$\mathcal{P}_n$ for $n=11,35,59,83,107.$}
	\end{table}

Berndt and Chan \cite[Theorem 1.2]{Bruce} later verified his claims for $n=11,35,59,83$, and $107$ using laborious  computations involving Greenhill polynomials and Weber class invariants, and proved that each $\mathcal{P}_n$ has $t_n$ as a root. However, due to computational complexity, their method to contruct $\mathcal{P}_n$ could not be applied for higher values of $n$. Thus, they asked for an efficient way of computing the polynomials $\mathcal{P}_n$ for every $n \equiv11\pmod{24}$. Moreover, the authors also proved the following:
\begin{thm}[{\cite[Theorem 4.1]{Bruce}}]\label{BB}
	Let $n \equiv 11 \pmod{24}$ be squarefree, and suppose that the class number of $\QQ(\sqrt{-n})$ is odd. Then $t_n$ is a real unit generating the Hilbert class field of $\QQ(\sqrt{-n})$. 
\end{thm}
Ten years later, Konstantinou and Kontogeorgis \cite{KK1} generalized this result by removing the constraint of the class number needing to be odd and also provided an efficient method for constructing the minimal polynomials $\mathcal{P}_n$ of $t_n$ over $\mathbb{Q}$ from the Ramanujan values $t_n$ for $n \equiv 11 \pmod{24}$, using the Shimura reciprocity law, and thus answered the demand made in \cite{Bruce} for a direct and an easily applicable construction method. Moreover, the authors also proved that the Ramanujan value $t_n$ is a class invariant for $n \equiv 11 \pmod{24}$ \cite[Theorem 3.4]{KK1}. Therefore, it follows that  degree of $\mathcal{P}_n$ equals the class number of the order of discriminant $-n$ 
for all positive $n \equiv 11 \pmod{24}$.

Let $\mathfrak{j}_n=j\left(\frac{-1+\sqrt{-n}}{2}\right)$, for all $n\in\NN$, where $j$ denotes the $j$-invariant as defined in Subsection \ref{sub}. In a follow-up paper in 2010, Konstantinou and Kontogeorgis proved the following:
\begin{prop}[{\cite[Lemma 3]{KKP2}}]\label{ROOT}
	Suppose $R_\mathcal{P}$ is a real root of a Ramanujan polynomial $\mathcal{P}_n$. Then, the real number $R_H$ obtained from the equation
	\[R_H=\left(R_{\mathcal{P}}^6 - 27R_\mathcal{P}^{-6}-6\right)^3,\]
	is a real root of the corresponding Hilbert class polynomial $H_n$.
\end{prop}
In their proof \cite[Pg. 12]{KKP2}, the authors show that, letting $R_{\mathcal{P}}=t_n$ gives $
	\mathfrak{j}_n=\left(t_n^6-27t_n^{-6}-6\right)^3.$

It is interesting to point out that coefficients of the polynomial $\mathcal{P}_n$ have remarkably smaller size compared to the coefficients of the corresponding Hilbert class polynomial $H_n$, which is a clear indication that their use in the CM method, which is used for the generation of elliptic curves over prime fields, can be especially favoured. Proposition \ref{ROOT} also suggests that the polynomials $\mathcal{P}_n$ can be used in the CM method because their roots can be transformed to the roots of $H_n$. For more details on constructing elliptic curves with the CM method, see \cite{A, Broker, Lay}.

In this brief article we study the discriminant of $\mathcal{P}_n$; the historical precedent for doing so comes from \cite{GZ}, which is known for computing the prime factorization of certain resultants of Hilbert class polynomials. Gross and Zagier \cite{GZ} also computed the prime factorization of the discriminant of the Hilbert class polynomial associated to the fundamental discriminant $-p$, where $p\equiv 3\pmod{4}$ is a prime. This result was later generalized by Dorman \cite{Dorman}, who extended the discriminant formula to the Hilbert class polynomials associated with arbitrary fundamental discriminants. Dorman's result in turn was then extended by Ye \cite{Ye}, who computed the prime factorization of Hilbert class polynomials associated to certain non-fundamental discriminants. 
For more details, see \cite{GZ, Dorman, Ye}.  
\subsection{Notations}
Fix a positive integer $n\equiv11\pmod{24}$ and let $K=\QQ\left(\sqrt{-n}\right)$ be an imaginary quadratic number field. Let $h_n$ and $\Cl\left(n\right)$ denote the class number and ideal class group of the order of discriminant $-n$, respectively, and let $\Cl\left(n\right)\left[2\right]$ be the subgroup of $\Cl\left(n\right)$ consisting of elements of order at most $2$. Note that $\ZZ\left[\mathfrak{j}_n\right]$ is contained in $\ZZ\left[t_n\right]$, which follows from Section \ref{S3}.
\subsection{Main results} Our first main result, which relates $\Delta\left(H_n\right)$ to $\Delta\left(\mathcal{P}_n\right)$, is as follows:
\begin{thm}\label{thm2}
	For all positive $n\equiv11\pmod{24}$, we have
	\begin{equation*}
		\Delta\left(H_n\right)=\Delta\left(\mathcal{P}_n\right)\left[\ZZ\left[t_n\right]:\ZZ\left[\mathfrak{j}_n\right]\right]^2,
	\end{equation*}
where  $\left[\ZZ\left[t_n\right]:\ZZ\left[\mathfrak{j}_n\right]\right]$ is the index of $\ZZ\left[\mathfrak{j}_n\right]$ in $\ZZ\left[t_n\right]$.
\end{thm}
\begin{rem}
	Since the quotient $\left[\ZZ\left[t_n\right]:\ZZ\left[\mathfrak{j}_n\right]\right]^2$ is a perfect square, we deduce that $	\Delta\left(H_n\right)$ and $\Delta\left(\mathcal{P}_n\right)$ have the same sign for all positive integers $n\equiv11\pmod{24}$.
\end{rem}
Let $\mathcal{D}\left(F\right)$ denote the discriminant of a number field $F$. Our next main result is the following:
\begin{thm}\label{thm1}
	For all positive integers $n\equiv11\pmod{24}$, we have
	\[\Delta\left(\mathcal{P}_n\right)=\left[\mathcal{O}_{\QQ\left(\mathfrak{j}_n\right)} : \ZZ[t_n]\right]^2,\] 
where  $\left[\mathcal{O}_{\QQ\left(\mathfrak{j}_n\right)} : \ZZ[t_n]\right]$ is the index of $\ZZ\left[t_n\right]$ inside $\mathcal{O}_{\QQ\left(\mathfrak{j}_n\right)}$.
\end{thm}  
\begin{rem}\label{rem}
	Note that Dorman explicitly computed $\mathcal{D}\left(\QQ\left(\mathfrak{j}_n\right)\right)$ in \cite{Dorman}. More precisely, he proved that, for a squarefree positive integer $n\equiv11\pmod{24}$, we have \cite[Proposition 5.1]{Dorman} $$\mathcal{D}\left(\QQ\left(\mathfrak{j}_n\right)\right)=D_0^{\frac{h_n}{2}}\cdot D_1^{\frac{h_n-2^{\mathfrak{t}-1}}{2}},$$ where where $\mathfrak{t}$ is the number of distinct prime factors of $n$, and $n=D_0 D_1$ with 
	\[D_1=\begin{cases}
		1 & \textrm{if at least } 2 \textrm{ primes congruent to } $3$ \textrm{ mod } 4 \textrm{ divide } D,
		\\
		p  & \textrm{if } p \textrm{ is the unique prime congruent to } 3 \textrm{ mod } 4 \textrm{ dividing } D.
	\end{cases}\]
\end{rem}
Next, we explicitly determines the sign of $\Delta\left(\mathcal{P}_n\right)$.
\begin{thm}\label{sign}
	For all positive $n\equiv11\pmod{24}$, $\Delta\left(\mathcal{P}_n\right)>0$ if and only if 
	\begin{equation*}
	h_n \equiv \left|\Cl\left(n\right)\left[2\right]\right| \pmod{4},
	\end{equation*}
where $\Cl\left(n\right)\left[2\right]$ is the subgroup of $\Cl\left(n\right)$ consisting of elements of order at most $2$.
\end{thm}
Note that using the ambiguous class number formula from genus theory, we find that for all positive squarefree inetgers  $n\equiv11\pmod{24}$, we have
\[\left|\Cl\left(n\right)\left[2\right]\right|=2^{\mathfrak{t}-1},\]
where $\mathfrak{t}$ is the number of distinct prime factors of $n$. Therefore, we have the following corollary:
\begin{cor}\label{ambicor}
		For all positive squarefree $n\equiv11\pmod{24}$,  $\Delta\left(\mathcal{P}_n\right)>0$ if and only if $$h_n\equiv2^{\mathfrak{t}-1}\pmod{4},$$
		where, as usual, $\mathfrak{t}$ denotes  the number of distinct prime factors of $n$.
\end{cor}

Moreover, when the ideal class group is a cyclic group, that is $\Cl\left(n\right)\cong \left(\ZZ/h_n\ZZ\right)$, which is true for all $n\equiv11\pmod{24}$, with $11\leqslant n \leqslant 995$ as illustrated in Table 2, it is easy to see that
\[\left|\Cl\left(n\right)[2]\right| =\begin{cases}
	1 & \textrm{if } h_n \textrm{ is odd}\\
	2 & \textrm{if } h_n \textrm{ is even.}
\end{cases}\]
Thus, we have the following corollary:
\begin{cor}\label{signcor}
	For all $n\equiv11\pmod{24}$, if $\Cl\left(n\right) \cong \left(\ZZ/h_n\ZZ\right)$, then $\Delta\left(\mathcal{P}_n\right)>0$ if and only if $$h_n\equiv1,2\pmod{4}.$$
\end{cor}
\begin{rem}
	For example, the two positive values of $n \equiv 11 \pmod{24}$ for which $\Cl\left(n\right) \not\cong \ZZ/h_n\ZZ$ are 1235 and 2555. In particular, we find that\footnote{The structure of the ideal class group $\Cl\left(n\right)$ for positive $n \equiv 11 \pmod{24}$ was computed using Sage \cite{Sg}.} $\Cl\left(1235\right)\cong\Cl\left(2555\right)\cong\left(\ZZ/6\ZZ\right) \times \left(\ZZ/2\ZZ\right)$. 
\end{rem}In Table 2, it can be observed that 3 seems to be the only prime that never appears in the prime factorization of $\Delta\left(\mathcal{P}_n\right)$ for all positive $n \equiv 11 \pmod{24}$. To prove this, we show that 3 never divides $\Delta\left(H_n\right)$, and thus $\Delta\left(\mathcal{P}_n\right)$, using Ye's explicit computation of $\Delta\left(H_n\right)$ \cite[Corollary 1.2]{Ye} for squarefree $n \equiv 11 \pmod{24}$, for more details see Section  \ref{S6}. 

Thus, we have the following important Theorem:
\begin{thm}\label{3}
	For all positive squarefree integers $n\equiv11\pmod{24}$, we have $3 \nmid \Delta\left(\mathcal{P}_n\right)$.
\end{thm}
\subsection{Organization of the paper} We start with a section on the necessary preliminaries and prove some important results that will later play a very crucial role in the proofs of our main results. Proofs of Theorem \ref{thm2}, \ref{thm1}, \ref{sign}, and \ref{3} are provided in Section  \ref{S3}, \ref{S4}, \ref{S5} and \ref{S6} respectively. In Table 2, we have computed the class number and the ideal class group structure of $\QQ(\sqrt{-n})$, the prime factorization and sign of $\Delta\left(\mathcal{P}_n\right)$ for all positive $n\equiv11\pmod{24}$ where $11\leqslant n \leqslant 995$ as an illustration of our main results.  Finally, in Section \ref{S7}, we show that all our main results hold true for $n=227$, as an example. All the computations were performed using Sage \cite{Sg}.
\section{Preliminaries}\label{S2}
\subsection{The $j$-invariant of a lattice}\label{sub}
A \emph{lattice} is defined to be an additive subgroup $L$ of $\CC$ which is generated by  two complex numbers $\omega_1$ and $\omega_2$ that are linearly independent over $\RR$. We express this by writing $L=[\omega_1,\omega_2].$ 
The $j$-invariant $j(L)$ of a lattice $L$ is defined to be the complex number 
\begin{equation*}
	j\left(L\right)=1728\dfrac{g_2\left(L\right)^3}{g_2\left(L\right)^3-27g_3\left(L\right)^2},
\end{equation*}
where
\begin{equation*}
	g_2\left(L\right)=60\sum_{\omega\in L\backslash\{0\} }^{}\dfrac{1}{\omega^4}=60\sum_{\substack{m,n=-\infty\\\left(m,n\right)\neq\left(0,0\right)}}^{\infty}\dfrac{1}{\left(m\tau+n\right)^{4}},
\end{equation*}
and
\begin{equation*}
	g_3\left(L\right)=140\sum_{\omega\in L\backslash\{0\} }^{}\dfrac{1}{\omega^6}=140\sum_{\substack{m,n=-\infty\\\left(m,n\right)\neq\left(0,0\right)}}^{\infty}\dfrac{1}{\left(m\tau+n\right)^{6}},
\end{equation*}
where $L=[1,\tau]$ with $\tau\in\HH$, the upper-half plane. 

\begin{prop}\label{CX1}
	Let $L$ be a lattice, and let $\overline{L}$ denote the lattice obtained by complex conjugation. Then $g_2(\overline{L}) = \overline{g_2(L)}, g_3(\overline{L}) = \overline{g_3(L)}$ and $j(\overline{L}) = \overline{j(L)}$. 	
\end{prop}
\begin{proof}
	From the definition of $g_2(L)$, we have
	\begin{align*}
		g_2(\overline{L})&=60\sum_{\omega\in\overline{L}\backslash\{0\}}^{}\dfrac{1}{\omega^4}
		=60\sum_{\omega\in L\backslash\{0\}}^{}\dfrac{1}{\overline{\omega^{4}}}
		=\overline{60\sum_{\omega\in L\backslash\{0\}}^{}\dfrac{1}{\omega^{4}}}
		=\overline{g_2(L)}.
	\end{align*}
A similar argument can be implemented to show that $g_3(\overline{L}) = \overline{g_3(L)}$. 

Finally, putting all things together produces
\[j\left(\overline{L}\right)=1728\dfrac{g_2\left(\overline{L}\right)^3}{g_2\left(\overline{L}\right)^3-27g_3\left(\overline{L}\right)^2}=1728\dfrac{\overline{g_2\left(L\right)^3}}{\overline{g_2\left(L\right)^3}-27\overline{g_3\left(L\right)^2}}=\overline{j(L)},\]
which is the desired result. 
\end{proof}
We say that two lattices $L$ and  $L'$ are \emph{homothetic} if there is a nonzero complex number $\lambda$ such that $L' = \lambda L$. Note that homothetic lattices have the same $j$-invariant.
\begin{lemma}[{\cite[Theorem 10.9]{Coz}}]\label{CX3}
	If $L$ and $L'$ are lattices in $\CC$, then $j(L) = j(L')$ if and only if both the lattices $L$ and $L'$ are homothetic.
\end{lemma}
Next, we prove an important result that will play a crucial role in the proof of Theorem \ref{sign}.
\begin{prop}\label{CX2}
	Let $\mathfrak{a}$ be a proper fractional $\mathcal{O}$-ideal, where $\mathcal{O}$ is an order in an imaginary quadratic number field. Then $j(\mathfrak{a})$ is a real number if and only if the class of $\mathfrak{a}$ has order at most 2 in the ideal class group $\Cl\left(\mathcal{O}\right)$.
\end{prop}
\begin{proof}
	From Proposition \ref{CX1}, we know that $j(\mathfrak{a})$ is a real number if and only if $j(\mathfrak{a})=j(\overline{\mathfrak{a}})$. Now Lemma \ref{CX3} tells us that this is  only possible when $\mathfrak{a}$ and $\overline{\mathfrak{a}}$ are homothetic. Or equivalently, when they represent the same ideal in the ideal class group $\Cl\left(\mathcal{O}\right)$. By computing $\mathfrak{a}\overline{\mathfrak{a}}$ and comparing the $j$-invariants by a similar argument that we just sketched above, it is easy to see that 
	\[\mathfrak{a}=\mathfrak{a}^{-1}=\overline{\mathfrak{a}}.\]
	Note that this is not true as a statement about ideals, here we are explicitly referring to ideal classes, that is, elements of $\Cl\left(\mathcal{O}\right)$. Putting all things together gives us the desired result. 
\end{proof}
As a consequence of the above Proposition, we have the following important corollary:
\begin{cor}\label{J}
	The $j$-invariant $j(\mathcal{O})$ is a real number for any order $\mathcal{O}$.
\end{cor}
\subsection{Discriminant of an algebraic number field}
Let $\mathcal{D}\left(K\right)$ denote the discriminant of an algebraic number field $K$. The discriminant of a nonzero finitely generated $\ZZ$-submodule $\mathfrak{a}$ of a number field $K$ is defined as
\begin{equation*}
	\mathcal{D}\left(\mathfrak{a}\right)=\mathcal{D}\left(\alpha_1,\alpha_2,\ldots, \alpha_n\right),
\end{equation*}
where $\mathfrak{a}=\ZZ\alpha_1+\ZZ\alpha_2+\cdots+\ZZ\alpha_n$ and spans $K$ as a $\QQ$-vector space.
\begin{prop}[{\cite[Prop. 2.12]{Neukirch}}]
	\label{index}
	If $\mathfrak{a} \subset \mathfrak{a}'$ are two nonzero finitely generated $\ZZ$-submodules of a number field $K$ that span $K$ as a $\QQ$-vector space, then the index $[\mathfrak{a}' : \mathfrak{a}]$ is finite and 
	\begin{equation*}
		\mathcal{D}\left(\mathfrak{a}\right)=\left[\mathfrak{a}':\mathfrak{a}\right]^2 \mathcal{D}\left(\mathfrak{a}'\right).
	\end{equation*}
\end{prop}
\begin{rem}
	Note that although Neukirch \cite[Prop. 2.12]{Neukirch} only states that Proposition \ref{index} is true for finitely generated $\mathcal{O}_K$-submodules, the same result and proof also works in the more general case, that is, for finitely generated $\ZZ$-submodules.
\end{rem}
\subsection{Number fields $\QQ\left(\mathfrak{j}_n\right)$ and  $\QQ\left(t_n\right)$}
The main goal of this subsection is to show that $\QQ\left(\mathfrak{j}_n\right)$ equals $\QQ\left(t_n\right)$. This important result indeed proves Theorem \ref{thm1} in a very simple and efficient way as explained later in Section \ref{S4}.
\begin{prop}\label{imp}
	For all positive integers $n\equiv11\pmod{24}$, we have $\QQ\left(\mathfrak{j}_n\right)=\QQ\left(t_n\right)$. 
\end{prop}
\begin{proof}
   To start with, consider the following field diagram:
	\begin{figure}[h]
		\label{Fig1}
		\centering
		\begin{tikzcd}
			&  & H                                                      &  &                                                             &  &                                                                            \\
			&  &                                                        &  &                                                             &  &                                                                            \\
			K=\mathbb{Q}\left(\sqrt{-n}\right) \arrow[rruu, no head, "h_n"] \arrow[rrdd, no head, "2"] &  &                                                        &  & \mathbb{Q}\left(\mathfrak{j}_n\right) \arrow[lluu, no head] &  & \mathbb{Q}\left(t_n\right) \arrow[lllluu, no head] \arrow[lllldd, no head] \\
			&  &                                                        &  &                                                             &  &                                                                            \\
			&  & \mathbb{Q} \arrow[rruu, no head] \arrow[uuuu, no head] &  &                                                             &  &                                                                           
		\end{tikzcd}
		\caption{}
	\end{figure}

We freely use some standard facts about the ring class fields of $K$ throughout the proof. 

From class field theory, recall that if $H$ is the ring class field associated to the discriminant of order $-n$, then $H$ is a finite Galois extension of $K$ and $[H : K] = h_n$, where $h_n$ denotes the class number of the order of discriminant $-n$. 
Therefore, we have $[H:K]=[H:\QQ(\sqrt{-n})]=h_n$. Next, it is easy to see that $[K:\QQ]=2$ since $\{1,\imath\sqrt{n}\}$ forms a basis. Moreover, since $H=\left(\QQ\left(\mathfrak{j}_n\right)\right)(\sqrt{-n}),$ we have $\left[H : \QQ\left(\mathfrak{j}_n\right)\right] \leqslant 2$. Theorefore, using the tower law, we deduce that either $\left[\QQ\left(\mathfrak{j}_n\right):\QQ\right]=h_n$ or $\left[\QQ\left(\mathfrak{j}_n\right):\QQ\right]=2h_n$. Now let $\mathfrak{f}$ be the minimal polynomial of $\mathfrak{j}_n \in H$ over $K$. Note that $\mathfrak{f}\in K[x]$ has degree $h_n$. From Corollary \ref{J}, we know that $\mathfrak{j}_n \in \RR$, that is $\overline{\mathfrak{j}_n}=\mathfrak{j}_n$, thus we have $\overline{\mathfrak{f}}=\mathfrak{f}$, which implies $\mathfrak{f}\in\QQ[x]$, and therefore, we deduce that $[\QQ\left(\mathfrak{j}_n\right):\QQ]\leqslant h_n$. 

Putting all things together finally produces $\left[\QQ\left(\mathfrak{j}_n\right):\QQ\right]=h_n$ and $\left[H:\QQ\left(\mathfrak{j}_n\right)\right]=2.$ Moreover, since $t_n \in \RR$, which easily follows from the Ramanujan's definition of $t_n$, a similar argument can be applied to show that $\left[\QQ\left(t_n\right):\QQ\right]=h_n$ and $\left[H:\QQ\left(t_n\right)\right]=2.$ 

All this discussion can be summarised by the following field diagram:
\begin{figure}[h]
	\centering
	\begin{tikzcd}
		&  & H                                                      &  &                                                             &  &                                                                            \\
		&  &                                                        &  &                                                             &  &                                                                            \\
		K=\mathbb{Q}\left(\sqrt{-n}\right) \arrow[rruu, no head, "h_n"] \arrow[rrdd, no head, "2"] &  &                                                        &  & \mathbb{Q}\left(\mathfrak{j}_n\right) \arrow[lluu, no head, "2"] &  & \mathbb{Q}\left(t_n\right) \arrow[lllluu, no head, "2"] \arrow[lllldd, no head, "h_n"] \\
		&  &                                                        &  &                                                             &  &                                                                            \\
		&  & \mathbb{Q} \arrow[rruu, no head, "h_n"] \arrow[uuuu, no head] &  &                                                             &  &                                                                           
	\end{tikzcd}
	\caption{}
\end{figure}

Now observe that we have the tower of fields $\QQ\subseteq\QQ\left(\mathfrak{j}_n\right)\subseteq\QQ\left(t_n\right)\subseteq H$, where  $\QQ\left(\mathfrak{j}_n\right)\subseteq\QQ\left(t_n\right)$ follows from Proposition \ref{ROOT}. Therefore, by the tower law, we have
\[\left[\QQ\left(t_n\right):\QQ\left(\mathfrak{j}_n\right)\right]\left[H:\QQ\left(t_n\right)\right]=\left[H:\QQ\left(\mathfrak{j}_n\right)\right] \implies \left[\QQ\left(t_n\right):\QQ\left(\mathfrak{j}_n\right)\right]=1,\]
which is only possible when $\QQ\left(\mathfrak{j}_n\right) =\QQ\left(t_n\right)$. This completes the proof. 
\end{proof}
\section{Proof of Theorem \ref{thm2}}\label{S3}
Since $t_n \in \ZZ\left[t_n\right]$ and the constant term of $\mathcal{P}_n$ is always $\pm1$, which follows from \cite[Section 3]{KKP2} where the authors state that $t_n$ is a unit, it follows that $t_n^{-1}\in\ZZ\left[t_n\right]$. From Proposition \ref{ROOT}, we have
\begin{equation*}
	\mathfrak{j}_n=\left(t_n^6-27t_n^{-6}-6\right)^3.
\end{equation*}
Therefore, we deduce that $\mathfrak{j}_n\in\ZZ\left[t_n\right]$ and thus  $
	\ZZ\left[\mathfrak{j}_n\right] \subseteq \ZZ\left[t_n\right]$. 
	Now plugging in $\mathfrak{a}=\ZZ\left[\mathfrak{j}_n\right]$ and $\mathfrak{a}'=\ZZ\left[t_n\right]$ in Proposition \ref{index} produces 
\[\mathcal{D}\left(\ZZ\left[\mathfrak{j}_n\right]\right)=\mathcal{D}\left(\ZZ\left[t_n\right]\right)\left[\ZZ\left[t_n\right] : \ZZ\left[\mathfrak{j}_n\right]\right]^2,\]
or equivalently,
\begin{equation*}
	\Delta\left(H_n\right)=\Delta\left(\mathcal{P}_n\right)\left[\ZZ\left[t_n\right]:\ZZ\left[\mathfrak{j}_n\right]\right]^2,
\end{equation*}
which is the desired result. \QED
\section{Proof of Theorem \ref{thm1}}\label{S4}
Simply notice that $\ZZ\left[t_n\right]\subseteq\mathcal{O}_{\QQ\left(t_n\right)}=\mathcal{O}_{\QQ\left(\mathfrak{j}_n\right)}\subseteq\QQ\left(\mathfrak{j}_n\right)$, where $\mathcal{O}_{\QQ\left(t_n\right)}=\mathcal{O}_{\QQ\left(\mathfrak{j}_n\right)}$ follows from the fact that $\QQ\left(\mathfrak{j}_n\right)=\QQ\left(t_n\right)$, as shown in Proposition \ref{imp}. Now, as we did earlier, plugging in $\mathfrak{a}=\ZZ\left[t_n\right]$ and $\mathfrak{a}'=\mathcal{O}_{\QQ\left(\mathfrak{j}_n\right)}$ in Proposition \ref{index} gives us the desired result. \QED 
\section{Proof of Theorem \ref{sign}}\label{S5}
From the definition of Hilbert class polynomial, recall that its roots are given by
\begin{equation*}
	\{j(\mathfrak{a}) \mid \mathfrak{a}\in\Cl\left(\mathcal{O}\right)\}.
\end{equation*}
From Proposition \ref{CX2}, we deduce that the real roots of Hilbert class polynomial $H_n$ are given by
\begin{equation*}
	\{j(\mathfrak{a}) \mid \mathfrak{a}\in\Cl\left(\mathcal{O}\right), \textrm{order of } \mathfrak{a} \textrm{ at most } 2\}=\Cl\left(n\right)[2].	
\end{equation*}
Therefore, the number of non-real roots of $H_n$ is given by  $\left|\Cl\left(n\right)\right|-\left|\Cl\left(n\right)\left[2\right]\right|=h_n-\left|\Cl\left(n\right)\left[2\right]\right|$.  Recall that the discriminant of a polynomial in $\QQ$ is positive if and only if the number of non-real roots of the polynomial is divisible by 4. Putting all things together gives us the desired result. \QED
\section{Proof of Theorem \ref{3}}\label{S6}
Assume that $n$ is a positive squarefree integer. Ye  \cite[Corollary 1.2]{Ye} very recently explicitly computed the discriminant of Hilbert class polynomials $H_n$ as
\begin{align*}
	&\log\left|\Delta\left(H_n\right)\right|\\&=-\dfrac{h_K}{4}\sum_{\substack{[\mathfrak{a}]\in\Cl_K\\ [\mathfrak{a}]\neq[\mathcal{O}_K]}}\sum_{\ell=0}^{n-1}\sum_{X,Y=-\infty}^{\infty}\kappa\left(1-\dfrac{n\left(2AX+BY\right)^2+\left(nY-2A\ell\right)^2}{4An},\dfrac{B\ell}{n}f_1^{(\mathfrak{a})}-\dfrac{2A\ell}{n}f_{2}^{(\mathfrak{a})}+L_{-}^{(\mathfrak{a})}\right)
\end{align*}
where $\mathfrak{a}=\left[A,\dfrac{B+\sqrt{-n}}{2}\right], f_1^{(\mathfrak{a})}=\begin{pmatrix}
	-1 & B \\ 0 & A
\end{pmatrix}, f_2^{(\mathfrak{a})}=\begin{pmatrix}
0 & C \\ 1 & 0
\end{pmatrix}, h_K$ is the class number of $K=\QQ\left(\sqrt{-n}\right)$, $L_{-}^{(\mathfrak{a})}$ denotes the lattice $\ZZ f_1^{(\mathfrak{a})}+\ZZ f_2^{(\mathfrak{a})}$, and 
\begin{align}\label{3(1)}
	&\kappa\left(m,\mu\right)\nonumber\\&=-\dfrac{1}{h_K}\left(\sum_{q \, \textrm{inert} }^{}\xi_q\left(m,\mu\right)\left(\ord_q(m)+1\right)\rho_K\left(mn/q\right)\log q + \rho_K\left(mn\right)\sum_{q \mid n}^{}\xi_q\left(m,\mu\right)\ord_q(mn)\log q\right) \nonumber
\end{align}
where $\xi_q$ and $\rho_K$ are  as defined in \cite[Corollary 1.2]{Ye}.

To prove Theorem \ref{3}, we will show that $\log3$ never appears  in $\kappa\left(m,\mu\right)$ for any choice of parameters $m$ and $\mu$.
From the definition of $\kappa\left(m,\mu\right)$, it suffices to prove that 3 always splits in $\QQ\left(\sqrt{-n}\right)$ and does not divide $n$ for all positive squarefree $n\equiv11\pmod{24}$.

Recall that a prime number $\mathfrak{p}$ splits in the imaginary quadratic  field $\QQ\left(\sqrt{-n}\right)$ if and only if $-n$ is a nonzero quadratic residue mod $\mathfrak{p}$. The quadratic residues mod 3 are 0 and 1, and it is easy to see that when $n\equiv11\pmod{24}$, we have $-n\equiv1\pmod{3}$. Thus $-n$ is a nonzero quadratic residue mod 3 for all positive squarefree integers $n\equiv11\pmod{24}$. This completes the proof. \QED
\section{An Exmaple of Main Results}\label{S7}
\begin{ex}
For $n=227$, we have \cite[Table 1]{KK1} $$\mathcal{P}_{227}(z)=z^5-5z^4+9z^3-9z^2+9z-1.$$
Now since 227 is squarefree, it follows that the degree of $\mathcal{P}_{227}$ is equal to the class number of $\QQ\left(\sqrt{-227}\right)$, and thus, $h_{227}=5$. The class group structure of $\QQ\left(\sqrt{-227}\right)$ is computed using Sage as $\ZZ/5\ZZ$, see Table 2. Now using Corollary \ref{signcor}, we have $\Delta\left(\mathcal{P}_{227}\right)>0$, since $$h_{227}=5\equiv1\pmod{4}.$$ 
From Table 2, we have $\Delta\left(\mathcal{P}_{227}\right)=2^4\cdot227^2$. 
Moreover, using the Sage database for Hilbert class polynomials and their discriminants \cite{Sg}, it can be  found that 
\begin{align*}H_{227}(z)=& \,z^5 + 360082897644683264000z^4 \\&- 2562327002832961536000000000z^3 \\&+ 18227340807938993794580480000000000z^2 \\&- 2111118203460821622718464000000000000z \\&+ 5085472193216544027705344000000000000000,
\end{align*}
and 
\[\Delta\left(H_{227}\right)=2^{316} \cdot 5^{60} \cdot 13^{20} \cdot 17^{20} \cdot 31^4 \cdot 37^6 \cdot 41^4 \cdot 61^4 \cdot 83^2 \cdot 151^2 \cdot 179^2 \cdot 191^2 \cdot 199^2 \cdot 227^2.\]
Therefore, we have 
\begin{align*}
	\dfrac{\Delta\left(H_{227}\right)}{\Delta\left(\mathcal{P}_{227}\right)}&=\left(2^{156} \cdot 5^{30} \cdot 13^{10} \cdot 17^{10} \cdot 31^2 \cdot 37^3 \cdot 41^2 \cdot 61^2 \cdot 83 \cdot 151 \cdot 179 \cdot 191 \cdot 199\right)^2,
\end{align*}
which verifies Theorem \ref{thm2}, since the quotient is also a perfect square. 

Now using Remark \ref{rem}, we have $D_0=1$, $D_1=227$, $\mathfrak{t}=1$, and thus
$$\mathcal{D}\left(\QQ\left(\mathfrak{j}_{227}\right)\right)=227^2 \mid 2^4\cdot227^2=\Delta\left(\mathcal{P}_{227}\right).$$ 
Moreover, notice that 3 does not appear in the prime factorization of $\Delta\left(\mathcal{P}_{227}\right)$ and $\Delta\left(H_{227}\right)$.
\end{ex}
\newpage
\FloatBarrier
\begin{table}[H]
	\centering
	\label{apx}
	\begin{tabular}{|l|l|l|l|l|}
		\hline \vspace{0.01in} 
		$n$   & $h_n$ & $\sgn(\Delta)$ & Prime factorisation of $\left|\Delta\left(\mathcal{P}_n\right)\right|$ & $\Cl\left(n\right)$  
		\vspace{0.01in}           \\ \hline
		$11$ & $1$ & $+$ & $1$ & $\ZZ/1\ZZ$ 
		\\
		$35$ & $2$ & $+$ & $5$  & $\ZZ/2\ZZ$ 
		\\
		$59$  & $3$    & $-$   & $59$                                                                                                      & $\ZZ/3\ZZ$   
		\\
		$83$  & $3$    & $-$   & $83$                                                                                                      & $\ZZ/3\ZZ$   
		\\
		$107$ & $3$    & $-$   & $107$                                                                                                    & $\ZZ/3\ZZ$   
		\\
		$131$ & $5$    & $+$   & $2^4 \times 131^2$                                                                                       & $\ZZ/5\ZZ$    
		\\
		$155$ & $4$    & $-$   & $2^4\times5^2\times31$                                                                                     & $\ZZ/4\ZZ$   
		\\
		$179$ & $5$    & $+$   & $2^6\times179^2$                                                                                          & $\ZZ/5\ZZ$   
		\\
		$203$ & $4$    & $-$   & $2^2\times7\times29^2$                                                                                   & $\ZZ/4\ZZ$    
		\\
		$227$ & $5$    & $+$   & $2^4\times227^2$                                                                                        & $\ZZ/5\ZZ$   
		\\
		$251$ & $7$    & $-$   & $2^{10}\times251^3$                                                                                       & $\ZZ/7\ZZ$    
		\\
		${275}$ & $4$    & $-$   & $2^2\times5^3\times7^2\times11$                                                                           & $\ZZ/4\ZZ$              
		\\
		$299$ & $8$    & $-$   & $2^{14}\times13^4\times23^3\times47^2$                                                                    & $\ZZ/8\ZZ$    
		\\
		$323$ & $4$    & $-$   & $2^2\times7^2\times17^2\times 19$                                                                        & $\ZZ/4\ZZ$    
		\\
		$347$ & $5$    & $+$   & $2^4\times347^2$                                                                                         & $\ZZ/5\ZZ$    
		\\
		$371$ & $8$    & $-$   & $2^{20}\times7^3\times17^2\times53^4$                                                                    & $\ZZ/8\ZZ$   
		\\
		$395$ & $8$    & $-$   & $2^{18}\times5^4\times13^4\times79^3$                                                                   & $\ZZ/8\ZZ$    
		\\
		$419$ & $9$    & $+$   & $2^{22}\times167^2\times 419^4$                                                                         & $\ZZ/9\ZZ$   
		\\
		$443$ & $5$    & $+$   & $2^6\times7^2\times443^2$                                                                                 & $\ZZ/5\ZZ$    
		\\
		$467$ & $7$    & $-$   & $2^{12}\times5^4\times467^3$                                                                             & $\ZZ/7\ZZ$    
		\\
		$491$ & $9$    & $+$   & $2^{28}\times7^2\times23^2\times491^4$                                                                   & $\ZZ/9\ZZ$    
		\\
		$515$ & $6$    & $+$   & $2^6\times5^3\times7^2\times13^2\times103^2$                                                            & $\ZZ/6\ZZ$    
		\\
		${539}$ & $8$    & $-$   & $2^{28}\times7^7\times11^4\times13^2$                                                                     & $\ZZ/8\ZZ$               
		\\
		$563$ & $9$    & $+$   & $2^{18}\times5^4\times311^2\times563^4$                                                                   & $\ZZ/9\ZZ$   
		\\
		$587$ & $7$    & $-$   & $2^{10}\times5^4\times13^4\times587^3$                                                                 & $\ZZ/7\ZZ$    
		\\
		$611$ & $10$   & $+$   & $2^{30}\times7^2\times13^9\times47^4$                                                                   & $\ZZ/{10}\ZZ$ 
		\\
		$635$ & $10$   & $+$   & $2^{28}\times5^5\times13^2\times127^4\times383^2$                                                        & $\ZZ/{10}\ZZ$ 
		\\
		$659$ & $11$   & $-$   & $2^{40}\times7^4\times191^2\times659^5$                                                                & $\ZZ/{11}\ZZ$ 
		\\
		$683$ & $5$    & $+$   & $2^6\times 7^2\times 683^2$                                                                             & $\ZZ/5\ZZ$   
		\\
		$707$ & $6$    & $+$   & $2^8\times7^2\times13^2\times19^2\times101^3$                                                             & $\ZZ/6\ZZ$   
		\\
		$731$ & $12$   & $-$   & $2^{52}\times17^6\times19^2\times43^5\times263^2\times479^2$                                              & $\ZZ/{12}\ZZ$
		\\
		$755$ & $12$   & $-$   & $2^{38}\times5^6\times41^4\times71^2\times 151^5\times503^2$                                             & $\ZZ/{12}$ 
		\\
		$779$ & $10$   & $+$   & $2^{40}\times19^8\times41^5\times311^2$                                                                   & $\ZZ/{10}\ZZ$ 
		\\
		$803$ & $10$   & $+$   & $2^{26}\times5^4\times11^4\times19^8\times73^5$                                                           & $\ZZ/{10}\ZZ$ 
		\\
		$827$ & $7$    & $-$   & $2^{12}\times7^2\times13^6\times827^3$                                                                   & $\ZZ/7\ZZ$    
		\\
		$851$ & $10$   & $+$   & $2^{44}\times7^2\times13^2\times23^4\times37^5\times167^2$                                              & $\ZZ/{10}\ZZ$   
		\\
		${875}$ & $10$   & $+$   & $2^{32}\times5^{13}\times7^{4}\times19^2\times37^4\times89^2$                                           & $\ZZ/{10}\ZZ$  
		\\
		$899$ & $14$   & $+$   & $2^{72}\times13^2\times19^4\times29^7\times31^6\times 647^2$                                            &  $\ZZ/{14}\ZZ$ 
		\\
		$923$ & $10$ & $+$ &  $2^{30} \times 5^4\times13^5\times19^2\times61^4\times71^4$ &  $\ZZ/{10}\ZZ$ 
		\\ 
		$947$ & $5$ & $+$ & $2^4\times13^2\times19^2\times947^2$ & $\ZZ/5\ZZ$ 
		\\ 
		$971$ & $15$ & $-$ & $2^{78}\times 41^4 \times 71^2\times503^2\times719^2\times971^7$  & $\ZZ/{15}\ZZ$ 
		\\
		$995$ & $8$ & $-$ & $2^{22}\times5^4\times7^2\times13^4\times19^2\times23^2\times199^3$ & $\ZZ/8\ZZ$ 
		\\\hline
	\end{tabular}
	\vspace{0.1in}
	\caption{Computation of the class number $h_n$, the sign $\sgn\left(\Delta\right)$ of $\Delta\left(\mathcal{P}_n\right)$, the structure of the class group $\Cl\left(n\right)$, and the prime factorization of $\left|\Delta\left(\mathcal{P}_n\right)\right|$ for all $n\equiv11\pmod{24}$ with $11\leqslant n \leqslant 995$.} 
\end{table} 
\FloatBarrier
\section{Concluding Remarks and Further Research}
We can now replace the Hilbert class polynomials $H_n$ and their discriminants $\Delta\left(H_n\right)$ with the Ramanujan polynomials $\mathcal{P}_n$ and their discriminants $\Delta\left(\mathcal{P}_n\right)$ to simplify computations wherever needed. 
Ramanujan polynomials $\mathcal{P}_n$ can be used in the generation of special curves, such as MNT curves \cite{OP1, OP2}, and in the generation of elliptic curves that do not necessarily have prime order \cite{A}.

Moreover, problems such as primality testing and proving \cite{A}, the generation of elliptic curve parameters \cite{KKP2} and the
representability of primes by quadratic forms \cite{Coz} could be considerably improved once we know more about Ramanujan polynomials $\mathcal{P}_n$ and the discriminant and traces. 
\section{Acknowledgements}
The majority of this research was done during the Research Science Institute (RSI) at MIT in the summer of 2022. First I would like to thank my mentor Alan Peng for his much-valued mentorship,  guidance and support, which included having daily meetings, sharing ideas, and clarifying any confusion, throughout the completion of this work. I would like to thank Prof. Andrew Sutherland from the MIT Math Department for suggesting this research project and Dr. Tanya Khovanova, Prof. Ankur Moitra, Prof. David Jerison and Dr. John Rickert for their helpful comments and suggestions on a rough draft of this manuscript, and for making this research project possible. I would also like to thank Prof. Aristides Kontogeorgis for giving us the pari-gp code to compute $\mathcal{P}_n$ for higher values of $n$. Thank you to the MIT and CEE for giving me the opportunity to attend the RSI and work on this project. Thank you to my sponsors, Mr. Nicholas Nash and Ms. Phalgun Raju, for making my participation in RSI possible. Finally, I would like to thank all my Rickoid friends who turned into a family and made my summer a memorable one. 


\begin{thebibliography}{10}	\bibitem{A}
	A. O. L. Atkin and F. Morain, Elliptic curves and primality proving. Math. Comp. \textbf{61} (1993), no. 203,
	29--68.
	\bibitem{Broker}
	Reinier Br\"oker and Peter Stevenhagen, Efficient CM-constructions of elliptic curves over
	finite fields, Mathematics of Computation \textbf{76} (2007), 2161--2179.
	\bibitem{Bruce}
	B. C. Berndt and H. H. Chan, Ramanujan and the modular $j$-invariant. Canad. Math. Bull.
	\textbf{42} (1999), 4, 427-440.
	\bibitem{Briggs}
	W. E. Briggs, An elementary proof of a theorem about the representations of primes by quadratic forms, Canadian J. Math. \textbf{6} (1954), pp. 353-363. 
	\bibitem{Coz}
	David A. Cox, Primes of the form $x^2 + ny^2$. John Wiley \& Sons Inc., New York, NY, 1989. 
	\bibitem{Dorman}
	Dorman, D.R. Singular moduli, modular polynomials, and the index of the closure of $\ZZ[j(\tau)]$ in $\QQ(j(\tau))$. Math. Ann. \textbf{283}, 177--191 (1989). 
	\bibitem{Gee} 
	A. Gee, Class invariants by Shimura’s reciprocity law. J. Th\'eor. Nombres Bordeaux 11(1999), 
	45-72.
	\bibitem{GZ}
	Gross, B.H., Zagier, D.B.: On singular moduli. J. Reine Angew. Math. \textbf{355}, 191--220 (1985).
	\bibitem{KK1}
	E. Konstantinou, A. Kontogeorgis, Computing Polynomials of the
	Ramanujan $t_n$ Class Invariants, Canad. Math. Bull. Vol. \textbf{52} (4), 2009.
	\bibitem{KK2}
	E. Konstantinou, A. Kontogeorgis, Y. Stamatiou, and C. Zaroliagis, Generating prime order elliptic
	curves: difficulties and efficiency considerations. In: International Conference on Information Security
	and Cryptology, Lecture Notes in Comput. Sci. 3506, Springer, Berlin, 2005, pp. 261-278.
	\bibitem{KKP2}
	Elisavet Konstantinou and Aristides Kontogeorgis, Ramanujan’s class invariants and their
	use in elliptic curve cryptography, Comput. Math. Appl. \textbf{59} (2010), no. 8, 2901--2917.
	\bibitem{Lay}
	Georg-Johann Lay and Horst G. Zimmer, Constructing elliptic curves with given group order over large finite fields, Algorithmic Number Theory Symposium--ANTS I (L. M. Adleman and
	M.D. Huang, eds.), Lecture Notes in Computer Science, \textbf{877}, 1994, pp. 250--263.
	\bibitem{OP1}
	A. Miyaji, M. Nakabayashi, S. Takano, New explicit conditions of elliptic curve traces for FR-reduction, IEICE Transactions on Fundamentals E84-A (\textbf{5})
	(2001) 1234--1243.
	\bibitem{Neukirch}
	J. Neukirch, Algebraic Number Theory, Springer Berlin, Heidelberg, 1999. 
	\bibitem{SR}
	S. Ramanujan, Notebooks. Vols. 1, 2, TIFR, Bombay, 1957.
	\bibitem{OP2}
	M. Scott, P.S.L.M. Barreto, Generating more MNT elliptic curves, Designs, Codes and Cryptography \textbf{38} (2006). 
	\bibitem{Silver}
	J. H. Silverman, Advanced topics in the arithmetic of elliptic curves, Graduate Texts in Mathematics
	\textbf{151}, Springer-Verlag, New York, 1994.
	\bibitem{Sg}
	Stein et al.: Sage Mathematics Software, Sage Development Team (2021). \url{http://www.sagemath.
	org}.
	\bibitem{Ye}
	Dongxi Ye, Revisiting the Gross-Zagier discriminant formula, Math. Nachrichten, \textbf{293}, 2020, 1801--1826
\end{thebibliography}
\end{document}